\documentclass[12pt]{article}
\usepackage{amsmath,amssymb,amsbsy,amsfonts,amsthm,latexsym,
amsopn,amstext,amsxtra,euscript,amscd}
\newtheorem{theorem}{Theorem}
\newtheorem{lemma}[theorem]{Lemma}

\newtheorem{cor}[theorem]{Corollary}
\newtheorem{prop}[theorem]{Proposition}

 \DeclareMathOperator{\lcm}{lcm}

\def\\{\cr}
\def\({\left(}
\def\){\right)}
\def\[{\left[}
\def\]{\right]}
\def\<{\langle}
\def\>{\rangle}

\def\Z{\mathbb{Z}}

\def\notdivides{\mathrel{\kern-3pt\not\!\kern3.5pt\bigm|}}

\begin{document}

\title{On a variant of Giuga numbers}

\author{
{\sc Jos\'{e} Mar\'{i}a Grau}\\
{Departamento de Matem\'{a}ticas}\\
{Universidad de Oviedo}\\
{Avda. Calvo Sotelo, s/n, 33007 Oviedo, Spain}\\
{grau@uniovi.es}\\
\\
{\sc Florian~Luca} \\
{Instituto de Matem{\'a}ticas}\\
{ Universidad Nacional Autonoma de M{\'e}xico} \\
{C.P. 58089, Morelia, Michoac{\'a}n, M{\'e}xico} \\
{fluca@matmor.unam.mx}\\
{and}\\
{ The John Knopfmacher Centre}\\
{for Applicable Analysis and Number Theory}\\
{ University of the Witwatersrand,  P.O.\ Wits 2050, South Africa}\\
\\
{\sc Antonio M. Oller-Marc\'{e}n}\\
{Departamento de Matem\'{a}ticas}\\
{Universidad de Zaragoza}\\
{C/Pedro Cerbuna 12, 50009 Zaragoza, Spain}\\
{oller@unizar.es}}

\maketitle

\newpage
\begin{abstract}
In this paper, we characterize the odd positive integers $n$ satisfying the congruence $\sum_{j=1} ^ {n-1} j^{ \frac{n-1}{2}}\equiv 0 \pmod n$.  We show that the set of such positive integers has an asymptotic density which turns out to be slightly larger than $3/8$.
\end{abstract}

\section{Introduction}

Given any property ${\mathbf P}$ satisfied by the primes, it is natural to consider the set ${\mathcal C}_{\mathbf P}:=\{n~{\text{\rm composite}}: n~{\text{\rm satisfies}} ~{\mathbf P}\}$. Elements of ${\mathcal C}_{\mathbf P}$ can be thought of as  pseudoprimes with respect to the property ${\mathbf P}$.  Such sets of pseudoprimes have been of interest to number theorists.

Putting aside practical primality tests such as Fermat, Euler, Euler--Jacobi, Miller--Rabin, Solovay--Strassen, and others, let us have a look at some interesting, although not very efficient, primality tests as summarized in the table below.

\medskip

 \begin{tabular}{|c|c|c|c|}
 \hline
    & Test &  Pseudoprimes & Infinitely many \\ \hline
1  & $(n-1)!\equiv-1\pmod n$ & None & No \\ \hline
2  & $ a^n \equiv a\pmod n$ {\text{\rm for~all}}  $a$ &   Carmichael numbers &Yes \\ \hline
3  & $\sum_{j=1} ^ {n-1} j^{\phi(n)}\equiv -1 \pmod n$ &  Giuga numbers & Unknown  \\ \hline
4  & $\phi(n) | (n-1) $& Lehmer numbers  & No example known \\ \hline
5 &  $\sum_{j=1} ^ {n-1} j^{ n-1}\equiv -1 \pmod n $ &    & No example known  \\ \hline
\end{tabular}

\medskip

In the above table, $\phi(n)$ is the Euler function of $n$.

The first test in the table, due to Wilson and published by Waring in \cite{Wil}, is an interesting and impractical characterization of a prime number. As a consequence, no pseudoprimes for this test exist.

The pseudoprimes for the second test in the table are called Carmichael numbers. They were characterized by Korselt in \cite{Kor}. In \cite{AGP}, it is proved that there are infinitely many of them. The counting function for the Carmichael numbers was studied by Erd\H os in \cite{Erd} and by Harman in \cite{Har}.

The pseudoprimes for the third test  are called Giuga numbers. The sequence of such  numbers is sequence A007850 in OEIS. These numbers were introduced and characterized in \cite{Bor}. For example, a Giuga number is a squarefree composite integer $n$ such that $p$ divides $n/p-1$ for all prime factors $p$ of $n$.  All known Giuga numbers are even. If an odd Giuga number exists, it must be the product of at least $14$ primes.  The Giuga numbers also satisfy the congruence $nB_{\phi(n)}\equiv-1\pmod n$, where for a positive integer $m$ the notation $B_m$ stands for the $m$th Bernoulli number.

The fourth test in the table is due to Lehmer (see \cite{Leh}) and it dates back to 1932. Although it has recently drawn much attention, it is still not known whether any pseudoprimes at all exist for this test or not. In a series of papers (see \cite{Pom1}, \cite{Pom2}, and \cite{Pom3}), Pomerance has obtained upper bounds for the counting function of the Lehmer numbers, which are the pseudoprimes for this test. In his third paper \cite{Pom3}, he succeded in showing that the counting function of the Lehmer numbers $n\le x$ is $O(x^{1/2} (\log x)^{3/4})$. Refinements of the underlying method of \cite{Pom3} led to subsequent improvements in the exponent of the logarithm in the above bound by Shan \cite{Sha}, Banks and Luca \cite{BaL}, Banks, G\"ulo\u glu and Nevans \cite{BGN}, and Luca and Pomerance \cite{LP}, respectively. The best exponent to date is due to Luca and Pomerance \cite{LP} and it is $-1/2+\varepsilon$ for any $\varepsilon>0$.

The last test in the table is based on a conjecture formulated in 1959 by Giuga \cite{Giu}, which states that the set of pseudoprimes for this test is empty. In \cite{Bor}, it is shown that every counterexample to Giuga's conjecture is both a Carmichael number and a Giuga number. Luca, Pomerance and Shparlinski \cite{LPS} have showed that the counting function for these numbers $n\le x$ is $O(x^{1/2}/(\log x)^{2})$ improving slightly on a previous result by Tipu \cite{Tip}.

In this paper, inspired by Giuga's conjecture, we study the odd positive integers $n$ satisfying the congruence
\begin{equation}
\label{eq:newgiuga}
\sum_{j=1}^{n-1} j^{(n-1)/2}\equiv 0 \pmod n.
\end{equation}
It is easy to see that if $n$ is an odd prime, then $n$ satisfies the above congruence. We characterize such positive integers $n$ and show that they have an asymptotic density which turns out to be slightly larger than $3/8$.

For simplicity we put
$$
G(n)=\sum_{j=1}^{n-1} j^{\lfloor (n-1)/2\rfloor},
$$
although we study this function only for odd values of $n$.

\section{On the congruence $G(n)\equiv 0\pmod n$ for odd $n$}

We put
$$
{\mathfrak P}:=\{n~{\text{\rm odd}}: G(n)\equiv 0\pmod n\}.
$$
It is easy to observe that every odd prime lies in ${\mathfrak P}$. In fact, by Euler's criterion, if $p$ is an odd prime, then ${\displaystyle{j^{(p-1)/2}\equiv \left(\frac{j}{p}\right)\pmod p}}$, where ${\displaystyle{\left(\frac{j}{p}\right)}}$ denotes the Legendre symbol of $j$ with respect to $p$. Thus,
$$
G(p)\equiv \sum_{j=1}^{p-1} \left(\frac{j}{p}\right)\equiv 0\pmod p,
$$
so that $p\in {\mathfrak P}$.

We start by showing that numbers which are congruent to $3\pmod 4$ are in ${\mathfrak P}$.

\begin{prop} If $n\equiv 3\pmod 4$, then $n \in {\mathfrak P} $.
\end{prop}
\begin{proof}
Writing $n=4m+3$, we have that $(n-1)/2=2m+1$ is odd. Now,
\begin{eqnarray*}
2G(n) & = & \sum_{j=1}^{n-1} \left(j^{2m+1}+(n-j)^{2m+1}\right)\\
& = & n\sum_{j=1}^{n-1} \left(j^{2m}+j^{2m-1}(n-j)+\cdots+(n-j)^{2m}\right),
\end{eqnarray*}
so $n\mid 2G(n)$. Since $n$ is odd, we get that $G(n)\equiv 0\pmod n$, which is what we wanted.
\end{proof}

The next lemma is immediate.

\begin{lemma}
\label{lem:1}
Let $p$ be an odd prime and let $k\geq1$ be an integer. Then
$$\gcd\left(\frac{p^k-1}{2},\varphi(p^k)\right)=\gcd\left(\frac{p^k-1}{2},p-1\right)=\begin{cases} p-1 & \textrm{if $k$ is even,}\\
(p-1)/2 & \textrm{if $k$ is odd.}
\end{cases}$$
\end{lemma}

With this lemma in mind we can prove the following result.

\begin{prop}
\label{prop:1}
Let $p$ be an odd prime and let $k\geq 1$ be any integer. Then, $p^k\in{\mathfrak P}$ if and only if $k$ is odd.
\end{prop}
\begin{proof}
Let $\alpha\in \Z$ be an integer whose class modulo $p^k$ is a generator of the unit group of $\Z/p^k\Z$.  We put $\beta:=\alpha^{(p^k-1)/2}$. Suppose first that $k$ is odd. We then claim that $\beta-1$ is not zero modulo $p$. In fact, if $\alpha^{(p^k-1)/2}\equiv 1\pmod p$, then since also $\alpha^{p-1}\equiv 1\pmod p$, we get, by Lemma \ref{lem:1}, that $\alpha^{(p-1)/2}\equiv 1\pmod p$, which is impossible.

Now, since $\beta-1$ is coprime to $p$, it is invertible modulo $p^k$. Moreover, since also $k\leq (p^k-1)/2$, we have that
\begin{eqnarray*}
G(n) & = & \sum_{j=1}^{n-1} j^{(p^k-1)/2}\equiv \sum_{\substack{\gcd(j,p)=1\\ 1\leq j\leq n-1}} j^{(p^k-1)/2}\pmod {p^k}\\
& \equiv & \sum_{j=1}^{\varphi(p^k)}\left(\alpha^{(p^k-1)/2}\right)^i\pmod {p^k}\equiv \sum_{i=1}^{\phi(p^k)} \beta^{i} \pmod {p^k}\\
& = & \frac{\beta^{\varphi(p^k)+1}-\beta}{\beta-1}\equiv 0\pmod {p^k}.
\end{eqnarray*}

Assume now that $k$ is even. Observe that
$$
(p^k-1)/2=(p-1)((1+p+\cdots+p^{k-1})/2):=(p-1)m,
$$
and $m$ is an integer which is coprime to $p$. Thus, $\beta=\alpha^{(p^k-1)/2}=(\alpha^{(p-1)})^m$ has order $p^{k-1}$ modulo $p^k$, and so does
$\alpha^{p-1}$. Moreover, again since $k\le (p^k-1)/2$, we may eliminate the multiples of $p$ from the sum defining $G(n)$ modulo $n$ and get
\begin{eqnarray}
\label{eq:kiseven}
G(n) & = & \sum_{j=1}^{n-1} j^{(p^k-1)/2}\equiv \sum_{\substack{\gcd(j,p)=1\\ 1\le j\le n-1}} j^{(p^k-1)/2}\pmod {p^k}\nonumber\\
& \equiv & \sum_{i=1}^{\varphi(p^k)}\left(\alpha^{(p^k-1)/2}\right)^i\equiv \sum_{i=1}^{p^{k-1}(p-1)}\left(\alpha^{(p-1)}\right)^{im}\pmod {p^k} \nonumber\\
& \equiv & (p-1)\sum_{i=1}^{p^{k-1}}\left(\alpha^{p-1}\right)^i\pmod {p^k}.
\end{eqnarray}
Since $\alpha^{p-1}$ has order $p^{k-1}$ modulo $p^k$, it follows that $\alpha^{p-1}=1+pu$ for some integer $u$ which is coprime to $p$. Then
\begin{equation}
\label{eq:2}
\sum_{i=1}^{p^{k-1}}\left(\alpha^{p-1}\right)^i=\alpha \left(\frac{\alpha^{p^{k-1}}-1}{\alpha-1}\right).
\end{equation}
Since $\alpha^{p^{k-1}}\equiv 1+p^k u\pmod {p^{k+1}}$, it follows that $(\alpha^{p^{k-1}}-1)/(\alpha-1)\equiv p^{k-1}\pmod {p^k}$, so that
\begin{equation}
\label{eq:3}
\alpha\left(\frac{\alpha^{p^{k-1}}-1}{\alpha-1}\right)\equiv \alpha p^{k-1}\pmod {p^k}\equiv p^{k-1}\pmod {p^k}.
\end{equation}
Calculations \eqref{eq:2} and \eqref{eq:3} together with congruences \eqref{eq:kiseven} give that $G(n)\equiv (p-1)p^{k-1}\pmod {p^k}$. Thus,
$p^k$ is not in ${\mathfrak P}$ when $k$ is even.
\end{proof}

Note that Proposition \ref{prop:1} does not extend to powers of positive integers having at least two distinct prime factors. For example,
$n=2021=43\times 47$ has the property that both $n$ and $n^2$ belong $\mathfrak{P}$.

\section{A characterization of $\mathfrak{P}$ and applications}

Here, we take a look into the arithmetic structure of the elements lying in $\mathfrak{P}$. We start with an easy but useful lemma.

\begin{lemma}
\label{lem:2}
Let $n=\prod_{p^{r_p}\| n} p^{r_p}$ be an odd integer, and let $A$ be any positive integer. If $\gcd(A,p-1)<p-1$ for all $p\mid n$, then
$$
\sum_{\substack{\gcd(j,n)=1\\1\leq j\leq n-1}}j^A\equiv 0\pmod n.
$$
\end{lemma}

\begin{proof}
It suffices to prove that the above congruence holds for all prime powers $p^{r_p}\| n$. So, let $p^{r}$ be such a prime power and let $\alpha$ be an
integer which is a generator of the unit group  of $\Z/p^{r}\Z$.
Put $\beta:=\alpha^A$. An argument similar to the one used in the proof of Proposition \ref{prop:1} (the case when $k$ is odd) shows that the condition
$\gcd(A,p-1)<p-1$ entails that $\beta-1$ is not a multiple of $p$. Thus, $\beta-1$ is invertible modulo $p$. We now have
\begin{eqnarray*}
\sum_{\substack{\gcd(j,n)=1\\ 1\le j\le n-1}} j^{A} & \equiv &  \left(\frac{\phi(n)}{\phi(p^{r})}\right)\sum_{\substack{\gcd(j,p)=1\\ 1\le j\le p}} j^{A}\pmod {p^r}
 \equiv  \phi(n/p^r)\sum_{i=1}^{\phi(p^r)} \alpha^{Ai}\pmod {p^r}\\
& \equiv & \phi(n/p^r)\sum_{i=1}^{\phi(p^r)} \beta^i\pmod {p^r} \equiv \phi(n/p^r)\frac{\beta^{\phi(p^r)+1}-\beta}{\beta-1}\pmod {p^r}\\
& \equiv &  0\pmod {p^r},
\end{eqnarray*}
which is what we wanted to prove.
\end{proof}

\begin{theorem}
\label{thm:1}
A positive integer $n$ is in ${\mathfrak P}$ if and only if $\gcd((n-1)/2,p-1)<p-1$ for all $p\mid n$.
\end{theorem}

\begin{proof}
Assume that $n$ is odd and $\gcd((n-1)/2,p-1)<p-1$. By Lemma \ref{lem:2},
$$
\sum_{\substack{(j,n)=1\\ 1\leq j\leq n-1}} j^{(n-1)/2}\equiv 0 \pmod n.
$$
Now, let $d$ be any divisor of $n$. Observe that
\begin{equation}
\label{eq:d}
\sum_{\substack{(j,n)=d\\ 1\leq j\leq n-1}}j^{\frac{n-1}{2}}=d^{\frac{n-1}{2}}\sum_{\substack{(i,n/d)=1\\ 1\leq i\leq n/d-1}}i^{\frac{n-1}{2}}.
\end{equation}
The last sum in the right--hand side of \eqref{eq:d} above is, by Lemma \ref{lem:2}, a multiple of $n/d$, so that the sum in the left--hand side of
\eqref{eq:d} above
is a multiple of $n$. Summing up these congruences over all possible divisors $d$ of $n$ and noting that
$$
G(n)=\sum_{d\mid n} \sum_{\substack{\gcd(j,n)=d\\ 1\le j\le n-1}} j^{(n-1)/2},
$$
we get that $G(n)\equiv 0\pmod n$, so $n\in {\mathfrak P}$.

Conversely, say $n\in {\mathfrak P}$ is some odd number and assume that there exists a prime factor $p$ of $n$ such that $p-1\mid (n-1)/2$. Write $(n-1)/2=(p-1)m$. Observe that $m$ is coprime to $p$. Assume that $p^{r}\|n$. Then, modulo $p^r$, we have
$$
G(n) =\sum_{j=1}^{n-1} j^{(n-1)/2}\equiv (n/p^r) \sum_{\substack{\gcd(j,p)=1\\ 1\le j\le p^r-1}} j^{(n-1)/2}\pmod {p^r}\equiv
(n/p^r)\sum_{\substack{\gcd(j,p)=1\\ 1\le j\le p^r-1}} j^{(p-1)}.
$$
The argument used in Proposition \ref{prop:1} (the case when $k$ is even), shows that the second sum is not zero modulo $p^r$, and since $n/p^r$ is also coprime to $p$, we get that $p^r$ does not divide $G(n)$, a contradiction.

This completes the proof of the theorem.
\end{proof}

Here are a few immediate corollaries of Theorem \ref{thm:1}.

\begin{cor}
Let $n$ be any integer. Assume that one of the following conditions hold:
\begin{itemize}
\item[i)] $\gcd\left((n-1)/2,\varphi(n)\right)$ is odd;
\item[ii)] $\gcd\left((n-1)/2,\lambda(n)\right)$ is odd, where $\lambda(n)$ the Carmichael function.
\end{itemize}
Then $n\in {\mathfrak P}$.
\end{cor}

\begin{cor}
If $n^k\in {\mathfrak P}$ for some $k\geq 1$, then $n\in {\mathfrak P}$.
\end{cor}

\begin{proof}
Observe that $\gcd\left((n-1)/2,p-1\right)$ divides $\gcd\left((n^k-1)/2,p-1\right)$ for every $k$ and every prime number $p$.
Now the corollary follows from Theorem \ref{thm:1}.
\end{proof}

We add another sufficient condition which is somewhat reminiscent of the characterization of the Giuga numbers.

\begin{prop}
\label{prop:2}
Let $n=\prod_{p^{r_p}\| n} p^{r_p}$ be an odd integer. If $p-1$ does not divide $n/p^{r_p}-1$ for every prime factor $p$ of $n$, then $n\in {\mathfrak P}$.
\end{prop}

\begin{proof}
By Theorem \ref{thm:1}, if $n\not\in {\mathfrak P}$, then there exists  a prime factor $p$ of $n$ such that $p-1$ divides $(n-1)/2$. In particular, $p-1\mid n-1$. Since $p-1$ also divides $p^{r_p}-1$, it follows that $p-1$ divides $n-p^{r_p}=p^{r_p}(n/p^{r_p}-1)$. Since $p-1$ is obviously coprime to $p^{r_p}$, we get that $p-1$ divides
$n/p^{r_p}-1$, which is a contradiction.
\end{proof}

It is also easy to determine whether numbers of the form $2^m+1$ are in ${\mathfrak P}$. Indeed, assume that $2^m+1\not\in {\mathfrak P}$ for some
positive integer $m$. Then, by Theorem
\ref{thm:1}, there is some prime $p\mid 2^m+1$ such that $p-1\mid ((2^m+1)-1)/2=2^{m-1}$. Thus, $p=2^a+1$ for some $a\le m-1$, and so $p$ is a Fermat prime. In particular, $a=2^{\alpha}$ for some $\alpha\ge 0$. Since $p=2^{2^{\alpha}}+1$ is a proper divisor of $2^m+1$, it follows that $2^{\alpha}\mid m$ and $m/2^{\alpha}$ is odd. This is possible only when $2^{\alpha}$ is the exact power of $2$ in $m$ and $m$ is not a power of $2$. So, we have the following result.

\begin{prop}
Let $n=2^m+1$ and $m=2^{\alpha} m_1$ with $\alpha\ge 0$ and odd $m_1>1$. Then $n\in {\mathfrak P}$ unless $2^{2^{\alpha}}+1$ is a Fermat prime.
\end{prop}

\section{Asymptotic density of $\mathfrak{P}$}

Let ${\mathbb{I}}$ be the set of odd positive integers. In order to compute the asymptotic density of ${\mathfrak P}$, or to even prove that it exists, it suffices to understand the elements in its complement $\mathbb{I}\backslash {\mathfrak P}$. It turns out that this is easy. For an odd prime $p$ let
$$
{\mathcal F}_p:=\{p^2 \pmod {2p(p-1)}\}.
$$
Observe that ${\mathcal F}_p\subseteq {\mathbb{I}}$.

\begin{theorem}
\label{thm:3}
We have
\begin{equation}
\label{eq:complement}
\mathbb{I}\backslash {\mathfrak P}=\bigcup_{p\ge 3} {\mathcal F}_p.
\end{equation}
\end{theorem}
\begin{proof}
By Theorem \ref{thm:1}, we have that $n\not\in {\mathfrak P}$ if and only if $p-1$ divides $(n-1)/2$ for some prime factor $p$ of $n$. This condition is equivalent to $n\equiv 1\pmod {2(p-1)}$. Write $n=pm$ for some positive integer $m$. Since $p$ is invertible modulo $2(p-1)$, it follows that $m$ is uniquely determined modulo $2(p-1)$. It suffices to notice that the class of $m$ modulo $2(p-1)$ is in fact $p$ since then $pm\equiv p^2\equiv 1\pmod {2(p-1)}$ with the last congruence following because
$p^2-1=(p-1)(p+1)$ is a multiple of $2(p-1)$. This completes the proof.
\end{proof}

Observe that ${\mathcal F}_p$ is an arithmetic progression of difference $1/(2p(p-1))$. Since the series
$$
\sum_{p\ge 3} \frac{1}{2p(p-1)}
$$
is convergent, it follows immediately that $\mathbb{I}\backslash {\mathfrak P}$; hence, also ${\mathfrak P}$,  has a density. This also suggests a way to compute the density of ${\mathfrak P}$ with arbitrary precision. Namely, say $\varepsilon>0$ is given. Let $3=p_1<p_2<\cdots$ be the increasing sequence of all the odd primes. Let $k:=k(\varepsilon)$ be minimal such that
$$
\sum_{j\ge k} \frac{1}{2p_j(p_j-1)}<\varepsilon.
$$
It then follows that numbers $n\not\in {\mathfrak P}$ which are divisible by a prime $p_j$ with $j\ge k$ belong to $\bigcup_{j\ge k} {\mathcal F}_{p_j}$, which is a set of density $<\varepsilon$. Thus, with an error of at most $\varepsilon$, the density of the set $\mathbb{I}\backslash {\mathfrak P}$ is the same as the density of
$$
\bigcup_{j<k} {\mathcal F}_{p_j},
$$
which is, by the Principle of Inclusion and Exclusion,
\begin{equation}
\label{eq:s}
\sum_{s\ge 1} \sum_{1\le i_1<i_2<\cdots<i_s\le k-1}\frac{\varepsilon_{i_1,i_2,\ldots,i_s}}{\lcm[2p_{i_1}(p_{i_1}-1),\ldots, 2p_{i_s}(p_{i_s}-1)]},
\end{equation}
with  the coefficient $\varepsilon_{i_1,i_2,\ldots,i_s}$ being zero if  $\bigcap_{t=1}^s {\mathcal F}_{p_{i_t}}=\emptyset$, and being $(-1)^{s-1}$ otherwise.
Taking $ \varepsilon:=0.00082$, we get that $k = 29$,
 $$
\rho(\bigcup_{j<29} {\mathcal F}_{p_j})=\frac{274510632303283394907222287246970994037}{2284268907516688397400621108446881752020}\approx 0.120174,
$$
and consequently  $\rho(\mathfrak{P})$ belongs to [0.379005, 0.379826]. So, we can say that
$$
\rho(\mathfrak{P}) = 0.379\ldots
$$
Here and in what follows, for a subset ${\mathcal A}$ of the set of positive integers we used $\rho({\mathcal A})$ for its density when it exists.

These computations were carried out with \emph{Mathematica}, for which it was necessary to have a good criterion to determine when the intersection of $\mathcal{ F}_p$ for various odd primes $p$ is empty. We devote a few words on this issue. Let us observe first that the condition $n\in {\mathcal F}_p$, which is equivalent
to the fact that $p\mid n$ and $p-1$ divides $(n-1)/2$, can be formulated as the pair congruences
\begin{eqnarray}
\label{eq:pair}
n & \equiv & 1\pmod {2(p-1)};\nonumber\\
n & \equiv & 0 \pmod p.
\end{eqnarray}
Assume now that ${\mathcal P}$ is some finite set of primes. Let us look at $\bigcap_{p\in {\mathcal P}} {\mathcal F}_p$. Put $m:=\prod_{p\in {\mathcal P}} p$. The first set of congruences \eqref{eq:pair} for all $p\in {\mathcal P}$ is equivalent to
\begin{equation}
\label{eq:cong1}
n\equiv 1  \pmod {2\lambda(m)},
\end{equation}
where $\lambda(m)=\lcm[p-1:p\in {\mathcal P}]$ is the Carmichael $\lambda$-function of $m$. The second set of congruences for $p\in {\mathcal P}$ is equivalent to
\begin{equation}
\label{eq:cong2}
n\equiv 0\pmod m.
\end{equation}
Since $1$ is not congruent to $0$ modulo any prime $q$, it follows that a necessary condition for  \eqref{eq:cong1} and \eqref{eq:cong2} to hold simultaneously is that $m$ and $2\lambda(m)$ are coprime. This is also sufficient by the Chinese Remainder Lemma in order for the pair of congruences \eqref{eq:cong1} and \eqref{eq:cong2} to have a solution $n$. Since $m$ is also squarefree, the condition that $m>1$ is odd and $m$ and $2\lambda(m)$ are coprime is equivalent to $m>2$ and $m$ and $\phi(m)$ are coprime. Put
\begin{equation}
\label{eq:M}
{\mathcal M}:=\{m>2: \gcd(m,\phi(m))=1\}.
\end{equation}
Thus, we proved the following result.

\begin{prop}
Let ${\mathcal P}$ be a finite set of primes and put $m:=\prod_{p\in {\mathcal P}} p$. Then $\bigcap_{p\in {\mathcal P}} {\mathcal F}_p$ is nonempty if and only if
$m\in {\mathcal M}$, where this set is defined at \eqref{eq:M} above. If this is the case, then the set $\bigcap_{p\in {\mathcal P}} {\mathcal F}_p$ is an arithmetic progression of difference $1/(2m\lambda(m))$.
\end{prop}

The condition that $m\in {\mathcal M}$ can also be formulated by saying that $m$ is odd, squarefree and $p\nmid q-1$ for all primes $p$ and $q$ dividing $m$.
We recall that the set ${\mathcal M}$ has been studied intensively in the literature. For example, putting ${\mathcal M}(x)={\mathcal M}\cap [1,x]$, Erd\H os \cite{Erd1} proved that
$$
\#{\mathcal M}(x)=e^{-\gamma} (1+o(1))\frac{x}{\log\log\log x}\qquad {\text{\rm as}}\quad x\to \infty.
$$
In particular, it follows that if ${\mathcal P}$ is a finite set of primes, then $\bigcap_{p\in  {\mathcal P}} {\mathcal F}_p\ne \emptyset$ if and only if ${\mathcal F}_p\bigcap {\mathcal F}_q\ne \emptyset$ for any two elements $p$ and $q$ of ${\mathcal P}$.

Finally, let us observe that with this formalism and the Principle of Inclusion and Exclusion, as in \eqref{eq:s} for example,  we can write that
$$
\rho({\mathfrak P})=\sum_{m\in {\mathcal M}\cup \{1\}} \frac{(-1)^{\omega(m)}}{2m\lambda(m)}.
$$
Here, $\omega(m)$ is the number of distinct prime factors of $m$. The fact that the above series converges absolutely follows easily from the inequality $\lambda(m)>(\log m)^{c\log\log\log m}$ which holds with some positive constant $c$ for all sufficiently large $m$ (see \cite{EPS}), as well the fact that the series
$$
\sum_{m\ge 2} \frac{1}{m(\log m)^2}
$$
converges. We give no further details.

\end{document}